\documentclass[12 pt]{amsart}

\usepackage{hyperref}
\usepackage{etex}
\usepackage[shortlabels]{enumitem}
\usepackage{amsmath}
\usepackage{amsxtra}
\usepackage{amscd}
\usepackage{amsthm}
\usepackage{amsfonts}
\usepackage{amssymb}
\usepackage{eucal}
\usepackage[all]{xy}
\usepackage{graphicx}
\usepackage{tikz-cd}
\usepackage{mathrsfs}
\usepackage{subfiles}
\usepackage{mathpazo}
\usepackage{euler}
\usepackage[colorinlistoftodos, textsize=tiny]{todonotes}
\usepackage{morefloats}
\usepackage{pdfpages}
\usepackage{thm-restate}
\usepackage[utf8]{inputenc}
\usepackage{epigraph}
\usepackage{csquotes}
\usepackage[margin=1in]{geometry}

\graphicspath{ {images/} }

\RequirePackage{color}
\definecolor{myred}{rgb}{0.75,0,0}
\definecolor{mygreen}{rgb}{0,0.5,0}
\definecolor{myblue}{rgb}{0,0,0.65}

\usepackage{hyperref}
\hypersetup{citecolor=blue}
\usepackage{tikz}
\usetikzlibrary{matrix,arrows,decorations.pathmorphing}

\theoremstyle{plain}
\newtheorem{theorem}{Theorem}[section]
\newtheorem{proposition}[theorem]{Proposition}
\newtheorem{lemma}[theorem]{Lemma}
\newtheorem{corollary}[theorem]{Corollary}
\theoremstyle{definition}

\newtheorem{remark}[theorem]{Remark}
\newtheorem{example}[theorem]{Example}

\newtheorem{question}[theorem]{Question}

\theoremstyle{remark}

\numberwithin{equation}{section}
  
\newcommand\nc{\newcommand}
\nc\on{\operatorname}
\nc\renc{\renewcommand}

\newcommand\bc{{\mathbb C}}

\newcommand\bq{{\mathbb Q}}
\newcommand\bp{{\mathbb P}}

\newcommand\bz{{\mathbb Z}}
\newcommand\ba{{\mathbb A}}

\newcommand\scc{\mathscr C}

\newcommand\sch{\mathscr H}

\newcommand \ra{\rightarrow}
\newcommand \spec{\text{Spec }}

\DeclareMathOperator\id{id}
\DeclareMathOperator\coker{coker}

\DeclareMathOperator\aut{Aut}

\DeclareMathOperator\im{im}

\DeclareMathOperator\pgl{PGL}
\DeclareMathOperator\gl{GL}

\DeclareMathOperator\ch{char}

\DeclareMathOperator\tp{top}

\DeclareMathOperator\GSp{GSp}

\def\listtodoname{List of Todos}
\def\listoftodos{\@starttoc{tdo}\listtodoname}

\title{A Lefschetz theorem for intersections with projective varieties}
\author{Aaron Landesman}

\begin{document}

\maketitle

\begin{abstract}
One version of the classical Lefschetz hyperplane theorem states that for $U \subset \mathbb P^n$ a smooth quasi-projective
variety of dimension at least $2$, and $H \cap U$ a general hyperplane section, the resulting map on \'etale fundamental groups
$\pi_1(H \cap U) \ra \pi_1(U)$ is surjective.
We prove a generalization, replacing the hyperplane by a general $\operatorname{PGL}_{n+1}$-translate
of an arbitrary projective variety:
If $U \subset \bp^n$ is a normal quasi-projective variety,
$X$ is a geometrically
irreducible projective variety of dimension at least $n + 1 - \dim U$, and $Y$ is a general $\operatorname{PGL}_{n+1}$-translate
of $X$,
then the map $\pi_1(Y \cap U) \ra \pi_1(U)$ is surjective.
\end{abstract}

\section{Introduction}
One version of the well known Lefschetz hyperplane theorem states that if we take a smooth projective scheme $U \subset \bp^n$
and intersect it with a plane $H$ of dimension at least $n+1-\dim U$, the map of \'etale fundamental groups $\pi_1(H \cap U) \ra \pi_1(U)$ is surjective
\cite[XII, Corollary 3.5]{sga2}.
There is also a generalization of this to quasi-projective schemes:
if we take a smooth quasi-projective scheme over the complex numbers $U \subset \bp^n$
and intersect it with a {\it general} plane $H$ of dimension at least $n+1-\dim U$, the map $\pi_1(H \cap U) \ra \pi_1(U)$ is surjective
\cite[Part II, Theorem 1.2]{goresky1988stratified}.
In this article, we further generalize the Lefschetz hyperplane theorem from a statement regarding the intersection with a general hyperplane to a statement
regarding the intersection with a general $\pgl_{n+1}$-translate of an arbitrary geometrically irreducible projective scheme of appropriate dimension.

For $B$ a $k$-scheme with a $\pgl_{n+1}(k)$ action
and $\sch \subset B \times \bp^n$ a closed subscheme, we say
$\sch$ is {\em $\pgl_{n+1}(k)$-invariant} if the $\pgl_{n+1}(k)$ action
on $B \times \bp^n$ (acting via automorphisms on $\bp^n$) 
sends $\sch$ to $\sch$.

\begin{theorem}
	\label{theorem:hilbert-scheme-galois-image}
	Let $k$ be any field of characteristic $0$ and let $U \subset \bp^n$
	be a quasi-projective normal geometrically connected scheme.
	Let $B$ be a $k$-scheme locally of finite type with a
$\pgl_{n+1}(k)$ action and $\sch \subset B \times \bp^n$ be a closed subscheme.
	Suppose further that the resulting map $\sch \ra B$ has geometrically irreducible fibers of dimension at least
	$n + 1- \dim U$ and is $\pgl_{n+1}(k)$-invariant.
	Then, there is a dense open subscheme $B^0 \subset B$
	so that for all $b \in B^0$, the map
	$\pi_1( \sch_b \cap U) \ra \pi_1(U)$ is surjective.
	\end{theorem}
	We complete the proof of this theorem in \autoref{subsubsection:proof-completion}.
	As an important special case of \autoref{theorem:hilbert-scheme-galois-image}, take
	$X \subset \bp^n$
	any geometrically irreducible projective scheme of dimension at least $n + 1 - \dim U$, 
	and $Y$ a general $\pgl_{n+1}(k)$-translate of $X$. In this case, the theorem then states that the map $\pi_1(Y \cap U) \ra \pi_1(U)$ is surjective.

	Note that changing the basepoint of $U$
	will not affect surjectivity of the map on fundamental groups, so we omit the basepoint
	from our notation.

\begin{remark}
	\label{remark:projective-version}
	A version of \autoref{theorem:hilbert-scheme-galois-image} holds in positive characteristic
	when $U$ is projective (as opposed to quasi-projective).
	More precisely, let $k$ be any field (of any characteristic) and let
	$U \subset \bp^n$ be any geometrically 
	connected normal projective scheme. Let $B$ be a $k$-scheme which is locally of finite type with a $\pgl_{n+1}(k)$ action,
	let $\sch \subset B \times \bp^n$ be a closed subscheme, 
	and let $\sch \ra B$ be a morphism which is $\pgl_{n+1}(k)$-invariant
	and has geometrically irreducible fibers of
	dimension at least
	$n+1 - \dim U$.
	Then, for every $b \in B$, $\pi_1(\sch_b \cap U) \ra \pi_1(U)$ is surjective.

	We first prove this in the case $k$ is algebraically closed:
	We only need verify that for any connected
	finite \'etale cover $W \ra U$ and every $b \in B$, the pullback $W \times_U (\sch_b \cap U)$ is connected.
	Observe that $W \times_U (\sch_b \cap U) \simeq W \times_{\bp^n} \sch_b$, and the latter is connected by	
	\cite[Corollary 7.3]{jouanolou1982theoremes}.
	We can deduce the case for general $k$ from the case that $k$ is algebraically closed
	from \cite[Expos\'e IX, Th\'eor\`eme 6.1]{noopsortSGA1Grothendieck1971}
	(the proof is essentially the same as that of \autoref{lemma:reduction-to-algebraically-closed-fields}).
	
	Thus, the main work in the proof of \autoref{theorem:hilbert-scheme-galois-image} comes in dealing with the quasi-projective assumption on $U$.
\end{remark}
\begin{remark}
	Suppose that the projective closure $\overline U$ of $U$
	is smooth and the maps $\sch \times_{\bp^n} U \ra B$ and $\sch \times_{\bp^n} U \ra U$ are smooth. 
	In this special case, \autoref{theorem:hilbert-scheme-galois-image}
	follows from \cite[Theorem 5]{kollar2015lefschetz},
	combined with the observation from \autoref{remark:projective-version},
	that \autoref{theorem:hilbert-scheme-galois-image} holds when applied to the projective scheme $\overline U$ in place of $U$.
	Note that in this case, the hypotheses $(2)$ and $(3)$ of \cite[Theorem 5]{kollar2015lefschetz}
	follow from the assumption that $\sch \ra B$ is $\pgl_{n+1}(k)$-invariant.
	In particular, 
	\cite[Theorem 5]{kollar2015lefschetz} applies in the interesting special case
	that $U \subset \bp^n$ is a dense open, $B$ is smooth, and the map $\sch \ra B$ is smooth.
Also, see \cite{kollar2000fundamental} and \cite{kollar2003rationally}
for related results.
\end{remark}
\begin{example}[Failure in characteristic $p$]
	\label{example:failure-in-char-p}
	We note that \autoref{theorem:hilbert-scheme-galois-image}
	does not hold over a field $k$ of positive characteristic,
	so the characteristic $0$ assumption is necessary.
	A counterexample is provided by the case that
	$U = \ba^2 \subset \bp^2$
	and
	$\sch \ra B$ is taken to be the universal family over the Grassmannian of lines in $\bp^2$.
	To show \autoref{theorem:hilbert-scheme-galois-image} does not hold in this case,
	we will show there is no closed point $b \in B$
	for which the map $\pi_1(\sch_b \cap \ba^2) \ra \pi_1(\ba^2)$
	is surjective.
	Since $\sch_b$ is a line in $\bp^2$, the intersection 
	$\sch_b \cap \ba^2$
	is either $\ba^1 \subset \ba^2$ embedded linearly, or empty.
	So, we only need show that the map
	$\pi_1(\ba^1) \ra \pi_1(\ba^2)$ is not surjective.
	To show this, we will produce a connected finite \'etale cover
	$W \ra \ba^2$ so that $W \times_{\ba^2} \ba^1$ is disconnected.
	If we choose coordinates so that the map $\ba^1 \ra \ba^2$ is given
	by $\spec k[x,y]/(y) \ra \spec k[x,y]$ then the Artin-Schreier
	cover $W := \spec k[x,y,t]/(t^p-t-y) \ra \spec k[x,y]$ does
	the trick.
\end{example}
\begin{remark}
	\autoref{example:failure-in-char-p} shows when $\ch k > 0$, 
	there may be no closed points
	$b \in B$ so that the map $\pi_1(X \times_{\bp^n} H) \ra \pi_1(X)$
	is surjective. 
	However, it is always true that
	for a generic point $\eta \in B$, the map
$\pi_1(X \times_{\bp^n} \sch_\eta) \ra \pi_1(X)$
	is surjective.
	This follows in the case that $k$ is algebraically closed from \autoref{corollary:general-pgl-translate} proven below,
	and can be deduced for arbitrary fields $k$ 
	from \cite[Expos\'e IX, Th\'eor\`eme 6.1]{noopsortSGA1Grothendieck1971}
	(the proof is essentially the same as that of \autoref{lemma:reduction-to-algebraically-closed-fields}).
	See also the related discussion in the case that the fibers of $\sch \ra B$
	are hyperplanes in the answers to the mathoverflow post
	\cite{MO:lefschetz-bertini}.
\end{remark}
The proof of \autoref{theorem:hilbert-scheme-galois-image} makes heavy use of the fact that we are working with the \'etale fundamental group
and not the topological fundamental group. This naturally leads to the following question.
\begin{question}
	\label{question:}
	Does \autoref{theorem:hilbert-scheme-galois-image} continue to hold over the complex numbers
	if the \'etale fundamental group is replaced by the topological fundamental group?
\end{question}

\subsection{Application to Galois Representations}
Given a scheme $U$ over a field of characteristic $0$ and an abelian scheme $f:A \to U$, let $A[n]$ denote the relative $n$-torsion of $f$. There is a Galois representation on the relative $n$-torsion $\rho_{n,A}: \pi_1(U) \to \aut(A[n]) \subset \GSp_{2g}(\bz/n\bz)$. Note that $\aut(A[n]) \subset \GSp_{2g}(\bz/n\bz)$ as opposed to only $\gl_{2g}(\bz/n\bz)$ because the action respects the symplectic form
on $A[n]$ given by the Weil pairing.
These $\mod n$ representations form a compatible system, ordered by divisibility, and taking the limit, we obtain a Galois representation
$\rho_{A}: \pi_1(U) \to \GSp_{2g}(\widehat{\bz})$
\autoref{theorem:hilbert-scheme-galois-image} has the following straightforward application to the study of Galois representations.

\begin{corollary}
	\label{corollary:galois-representation-surjectivity}
	Let $k$ be a field of characteristic $0$.
	Let $f\colon A \ra U$ be an abelian scheme of relative dimension $g$
	over $U$ with $U \subset \bp^n_k$ geometrically connected, normal, and quasi-projective. 
	Let $B$ be a $k$-scheme locally of finite type with a $\pgl_{n+1}(k)$
action and let $\sch \subset B \times \bp^n$ be a closed subscheme
	so that $\sch \ra B$ is a
	$\pgl_{n+1}(k)$-invariant morphism with geometrically irreducible fibers of dimension at least $n+1 - \dim U$.
	Let $\rho_A \colon \pi_1(U) \ra \GSp_{2g}(\widehat{\bz})$ denote the corresponding Galois representation.
	Then,
	for a general $b \in B$, $\im \rho_A = \im \rho_{f^{-1}(\sch_b \cap U)}$.
\end{corollary}
\begin{proof}
	By \autoref{theorem:hilbert-scheme-galois-image}, we know that for a general $b \in B$, $\iota_b\colon \pi_1(U \cap \sch_b) \ra \pi_1(U)$ is surjective.
	Since $\rho_A \circ \iota_b= \rho_{f^{-1}(\sch_b \cap U)}$, the result follows.
\end{proof}

\begin{remark}
	\autoref{corollary:galois-representation-surjectivity} offers a function field version of~\cite[Theorem 1.1]{landesman-swaminathan-tao-xu:rational-families}.
That is, ~\cite[Theorem 1.1]{landesman-swaminathan-tao-xu:rational-families} states that for an abelian scheme $f \colon A \ra U$ over a number field $k \neq \bq$
with $U$ rational, the Galois image of $A$ agrees with that of most $K$-points of $A$, counted by height.
Here, when we replace the number field by the function field of a curve, we obtain that the Galois image of $A$ agrees with that of
$f^{-1}(U \cap \sch_b)$, for $b \in B$ general. 
Note that in this function field case, unlike in the number field case, we make no hypotheses on whether $U$ is rational.
Further we obtain that there is a {\it Zariski}-open locus where the Galois images of the specializations agree with that of $A$, 
while in the number field case, the corresponding locus is typically not
Zariski-open, as is discussed in \cite[Remark 1.3]{landesman-swaminathan-tao-xu:rational-families}.
\end{remark}

\subsubsection{}
Further, \autoref{corollary:galois-representation-surjectivity} offers potential applications to studying images of Galois representations
under specialization.
For example, by \cite[Theorem 1.1]{cadoretuniform-ii}, any curve has a Zariski open locus of $k$-points (for $k$ a number field)
on which the index of the monodromy of the $k$-points in the monodromy of the family is finite.
However, it is an open question as to whether a higher dimensional base has a Zariski-open locus of $k$-points on which the monodromy groups have finite index
in $\rho_A(\pi_1(U))$.
Since \autoref{corollary:galois-representation-surjectivity} relates the monodromy of $A$ to the monodromy of curves on $A$,
in which case it is known there are only finitely many rational points with infinite index,
\autoref{corollary:galois-representation-surjectivity} may help in understanding the locus where the monodromy has infinite index.

\subsection{Outline of proof}

We now outline the proof of \autoref{theorem:hilbert-scheme-galois-image}, which will occupy the remainder of the paper.

Using standard techniques, we reduce to the case that $k = \bc$ in \autoref{subsection:reduction}.
The key issue is to prove \autoref{theorem:hilbert-scheme-galois-image} in the case $k = \bc$ and that $\sch_b \cap U$ is a curve. 
To check surjectivity, we verify that for every connected finite \'etale cover $W \ra U$ and a general $b \in B$, the pullback $W \times_U \sch_b$ is connected.
In \autoref{subsection:finite-quotients}, we show that for a {\it single} cover $W \ra U$, the pullback to a general member $\sch_b$ is connected.
To prove this we show that the pullback of $\sch_b \times_{\bp^n} W \ra \bp^n \times \bp^n$ along the diagonal $\Delta: \bp^n \ra \bp^n \times \bp^n$
is connected.
This connectivity is verified using several variants of Bertini's theorem:
One variant allows us replace the setup in $\bp^n$ with an analogous one in $\ba^n$.
Another variant verifies the corresponding statement in $\ba^n$.

However, the above argument only shows that for a fixed connected finite \'etale cover, the pullback of a general member is connected. 
To prove the theorem, we need to show
that for {\em every} connected finite \'etale cover, the pullback of a general member is connected. 
We show this in \autoref{subsection:completion-of-proof}.
The key input is Ehresmann's fibration theorem which  
shows that there is a dense open subscheme of our family which is locally trivial over $\bc$ in the $C^\infty$ topology.
By deforming one member of this dense open to all others, we deduce that if one member of this dense open surjects onto a given
finite quotient of $\pi_1(U)$ then all members of this dense open surject onto that finite quotient.

\section{Proof of main result}
\label{section:proof}
\subsection{Surjections onto finite quotients}
\label{subsection:finite-quotients}

We now prove the main technical tool of the paper, \autoref{proposition:general-fixed-plane-translate}.
From this, we deduce
\autoref{corollary:general-pgl-translate}, which, in the notation of \autoref{theorem:hilbert-scheme-galois-image}, implies
that for a fixed geometrically connected finite \'etale cover of $U$, the pullback to $U \cap \sch_b$ for $b \in B$ general is geometrically irreducible.
\autoref{corollary:general-pgl-translate} will be used in proving
\autoref{theorem:hilbert-scheme-galois-image} over $\bc$ in \autoref{proposition:case-over-c}.
\begin{proposition}
	\label{proposition:general-fixed-plane-translate}
	Let $k$ be an arbitrary field and let
	$U$ be a normal geometrically connected $k$-scheme with an embedding $U \subset \bp^n$ and with closure $\overline U \subset \bp^n$.
	Let $W \ra U$ be a geometrically connected finite \'etale cover.
	Let $Y$ 
	denote any geometrically irreducible closed subscheme of $\bp^n$ with $\dim Y = n + 1 - \dim U$.
	Let $H$ denote a hyperplane in $\bp^n$
	intersecting $Y \cap U$ and $Y \cap \overline U$ in a finite, nonempty set of points
	and let $G_H \subset \pgl_{n+1}(k)$ denote the subgroup of
	$\pgl_{n+1}(k)$ fixing $H$. 
	Then, for a general $\sigma \in G_H$,
	$\sigma(Y) \times_{\bp^n} W$
	is geometrically irreducible.
	\end{proposition}
\begin{proof}
Let $\overline W$ denote the normalization of
$\overline U$ in $K(W)$.
We now reduce to showing that $\sigma(Y) \times_{\bp^n} \overline W$ is geometrically irreducible for general $\sigma \in G_H$.
Further, observe that $W = \overline W \times_{\overline U} U$, since normalization respects base change. 
In particular, $W \subset \overline W$ a dense open subscheme.
Since $Y \cap U \cap H$ is nonempty, for a general $\sigma\in G_H$, we have that $\sigma(Y) \times_{\bp^n} W \subset \sigma(Y) \times_{\bp^n} \overline W$
is open and nonempty.
Thus, in order to show $\sigma(Y) \times_{\bp^n} W$ is geometrically irreducible, it suffices to show $\sigma(Y) \times_{\bp^n} \overline W$
is geometrically irreducible.

We next reduce to showing that $(\sigma(Y) \times_{\bp^n} \overline W) \times_{\bp^n} (\bp^n - H)$
is geometrically irreducible
for general $\sigma \in G_H$.
Since $\dim (\sigma(Y) \times_{\bp^n} \overline W) \times_{\bp^n} H = 0$ for a general
$\sigma \in G_H$ by assumption,
and
$\dim \sigma(Y) \times_{\bp^n} \overline W \geq 1$ for all $\sigma \in G_H$,
it follows that $\dim \sigma(Y) \times_{\bp^n} \overline W =1$ for a general $\sigma \in G_H$.
Therefore, $\sigma(Y) \times_{\bp^n} \overline W$
is $1$-dimensional for a general $\sigma \in G_H$.
Since $\overline W \ra \overline U$ is finite,
hence proper,
it follows from
\cite[Corollary 7.3]{jouanolou1982theoremes}
that $\sigma(Y) \times_{\bp^n} \overline W$ is geometrically connected.
Note that any irreducible component of a $1$-dimensional connected scheme
must be $1$-dimensional.
Since for a general $\sigma \in G_H$,
$(\sigma(Y) \times_{\bp^n} \overline W) \cap H$
is a finite collection of points,
if 
$(\sigma(Y) \times_{\bp^n} \overline W) \cap (\bp^n - H)$
is geometrically irreducible, 
$\sigma(Y) \times_{\bp^n} \overline W$ must also be geometrically irreducible.

To complete the proof, it suffices to show that for a general $\sigma \in G_H$, $(\sigma(Y) \times_{\bp^n} \overline W) \cap (\bp^n - H)$
is geometrically irreducible.
This holds by the following \autoref{lemma:affine-space-intersection-irreducibility}
by taking $X$ to be $Y \times_{\bp^n} (\bp^n - H)$ and $Z$ to be $\overline W \times_{\bp^n} (\bp^n - H)$.
\end{proof}
To complete the proof of \autoref{proposition:general-fixed-plane-translate},
it suffices to prove the following lemma.
\begin{lemma}
	\label{lemma:affine-space-intersection-irreducibility}
	Choose a hyperplane $H \subset \bp^n$ and let $\ba^n := \bp^n - H$.
	Let $G_H$ denote the automorphisms of $\bp^n$ fixing $H$,
	which then acts on $\ba^n = \bp^n - H$.
	Let $X \subset \ba^n$ be a closed subscheme 
	and let $\psi \colon  Z \ra \ba^n$ be a morphism
	so that $X$ and $Z$ are geometrically irreducible
	and $\dim X + \dim \psi(Z) = n + 1$.
	Then, for a general $\sigma \in G_H$,
	$\sigma(X) \times_{\ba^n} Z$ is geometrically irreducible.
\end{lemma}
\begin{proof}
Choose a general $n$-dimensional affine subspace $S \simeq \ba^n$ with an inclusion
\begin{align*}
	\iota  \colon  S & \rightarrow \ba^n \times \ba^n \\
	\left( x_1, \ldots, x_n \right) & \mapsto \left( f_1(x_1, \ldots, x_n), \ldots, f_{2n}(x_1, \ldots, x_n) \right)
\end{align*}
where $f_i$ are polynomials of the form $k_0 + k_1 x_1 + \cdots k_n x_n$.
We may further assume that $f_{n+i}(x_1, \ldots, x_n) = x_i$ for $1 \leq i \leq n$,
since a general $n$-plane can be written in this form.
For such a general such $n$-plane,
we can find $\sigma \in G_H$ so that $\iota$ factors as
\begin{equation}
	\label{equation:hyperplane-factorization}
	\begin{tikzcd}
		S \ar {rr}{\Delta} \ar {rd}{\iota} && S \times S \ar {ld}{\left( \sigma, \id \right)} \\
		 & \ba^n \times \ba^n. & 
	 \end{tikzcd}\end{equation}

Define $\Gamma$ to be the fiber product
\begin{equation}
	\nonumber
	\begin{tikzcd} 
		\Gamma \ar {r} \ar {d} & S \ar {d}{\iota} \\
		X \times Z \ar {r} & \ba^n \times \ba^n.
	\end{tikzcd}\end{equation}
Using the identification of $\iota$ with $(\sigma, \id)$ from
\eqref{equation:hyperplane-factorization},
the square
\begin{equation}
	\nonumber
	\begin{tikzcd} 
		\Gamma \ar{r} \ar{d}& S  \ar {d}{\Delta}  \\
		X \times Z \ar{r}{(\sigma, \id)} &	\ba^n \times \ba^n
	\end{tikzcd}\end{equation}
is also Cartesian.
This implies that
$\Gamma = \sigma(X) \times_{\ba^n} Z = \sigma(X) \cap Z.$
Summarizing, for a general $\sigma \in G_H$,
we obtain $\sigma(X) \times_{\ba^n} Z$
as the pullback of $X \times Z \rightarrow \ba^n \times \ba^n$
along a general $n$-plane $S \xrightarrow \iota \ba^n \times \ba^n$.
So, by Bertini irreducibility
\cite[Corollary 6.7(3)]{jouanolou1982theoremes}, we obtain that $\sigma(X) \times_{\ba^n} Z$ is geometrically irreducible for
a general such $\sigma$, as desired.
\end{proof}

We now use constructibility of the geometrically irreducible locus of a morphism to bootstrap the result of \autoref{proposition:general-fixed-plane-translate}
to apply to arbitrary $\pgl_{n+1}(k)$-invariant families.
\begin{corollary}
	\label{corollary:general-pgl-translate}
	Let $k$ be an arbitrary algebraically closed field and let
	$U \ra \spec k$ be a normal connected scheme with an embedding $U \subset \bp^n$. 
	Let $W \ra U$ be a connected finite \'etale cover.
Let $B$ be a $k$-scheme, locally of finite type with a $\pgl_{n+1}(k)$ action.
	Let
	$\sch \subset B \times \bp^n$ be a closed subscheme
	so that $\sch \rightarrow B$ has geometrically irreducible fibers of dimension
	$n + 1 - \dim U$
	and is $\pgl_{n+1}(k)$-invariant.
	Then, there is a dense open subscheme $B^0 \subset B$ so that for every $b \in B^0$,
	$\sch_b \times_{\bp^n} W$ is geometrically irreducible.
\end{corollary}
\begin{proof}
	By \cite[Th\'eor\`eme 9.7.7(iv)]{EGAIV.3}
	the locus of $B$ over which $\sch \times_{\bp^n} W$ is geometrically irreducible is constructible.
	Thus, to show there is an open locus over which $\sch \times_{\bp^n} W$ is geometrically irreducible, it suffices to show
	that for each $\pgl_{n+1}(k)$ orbit, there is an open locus over which $\sch \times_{\bp^n} W$ is geometrically irreducible.
	Therefore it suffices to prove the result in the case that $B \simeq \pgl_{n+1}(k)$. 
	By Bertini's theorem, for a general $b$, $\sch_b \cap U$ is $1$-dimensional,
	and so we may as well choose the isomorphism $B \simeq \pgl_{n+1}(k)$
	so that $\sch_e \cap U$
	is $1$-dimensional (where $e \in \pgl_{n+1}$ is the identity). 
	
	For $H \subset \bp^n$ a hyperplane, let $G_H \subset \pgl_{n+1}(k)$ denote the subscheme
	fixing $H$. We next show that every closed point of $\pgl_{n+1}(k)$ is contained in $G_H$ for some $H$.
	To see this, let $V$ denote an $(n+1)$-dimensional vector space, and choose an automorphism $M$ of $\bp V$,
	corresponding to a given element in $\pgl_{n+1}(k)$. Letting $V^\vee$ denote the dual vector space,
	$M$ induces a dual automorphism 
	$M^\vee$ on $\bp V^\vee$. Viewing $M^\vee$ as a matrix
	acting on $V^\vee$, since $k$ is algebraically closed, $M^\vee$ has an eigenvalue.
	This eigenvalue corresponds to a hyperplane $H$ fixed by $M$, and so $M \in G_H$.

	Let $\widetilde{B^0} $ denote the set of points such that $\sch_b \times_{\bp^n} W$ is geometrically irreducible.
Let $\overline U \subset \bp^n$ denote the closure of $U$.
	By Bertini's theorem, there is a dense open subset $S \subset \bp V^\vee$ parameterizing those $[H] \in \bp V^\vee$ 
	so that both $H \cap \sch_e \cap \overline U$
	and $H \cap \sch_e \cap U$ 
	are $0$-dimensional and nonempty.
	Since every closed point of $B$ is contained in $G_H$ for some $H$,
	it follows that $\cup_{[H] \in S} G_H$ is dense in $B$. 
	By \autoref{proposition:general-fixed-plane-translate} that $\widetilde{B^0}$ contains the generic point
	of $G_H$ for each $H$ in the dense open $S \subset \bp V^\vee$,
	and so $\widetilde{B^0}$ is also dense in $B$.
	Since $\widetilde{B_0} $ is constructible, it must contain a dense open subset $B^0 \subset B$.
\end{proof}

\subsection{Completion of proof over $\bc$}
\label{subsection:completion-of-proof}

The main result of this section is \autoref{proposition:case-over-c}, proving our main theorem over $k = \bc$.
Before proving this, we record an elementary lemma which will enable us to make a useful reduction at the beginning of
the proof of \autoref{proposition:case-over-c}.

\begin{lemma}
	\label{lemma:reduction-to-curves}
For a fixed field $k$, if \autoref{theorem:hilbert-scheme-galois-image} holds over $k$ in the case that the general fiber of $\sch \ra B$ has dimension $n + 1 - \dim U$,
then 
	\autoref{theorem:hilbert-scheme-galois-image} holds over $k$ in the case that the fibers of $\sch \ra B$ have dimension at least
	$n + 1 - \dim U$.
\end{lemma}
\begin{proof}
	We may first assume $B$ is integral, by passing to the reduction and considering each irreducible component separately.
	Let $\sch \ra B$ be a $\pgl_{n+1}(k)$-invariant morphism whose general fiber has dimension $n + 1 - \dim U + d$.
	Choose a codimension-$d$ plane $H$ so that for a general $b \in B$,
	$\dim H \cap \sch_b = n + 1 - \dim U$.
We want to show that for general $b \in B$, $\pi_1(\sch_b \cap U) \ra \pi_1(U)$ is surjective.
We know $\pi_1(H \cap \sch_b \cap U) \ra \pi_1(H \cap U)$ is surjective for a general $b \in B$, by applying the theorem 
to the family $H \times_{\bp^n} \sch \ra B$ 
of relative dimension $n + 1 - \dim U$
in $H \simeq \bp^{n-d}$.

Furthermore, we claim that for general $H$, the map $\pi_1(H \cap U) \ra \pi_1(U)$ is a surjection.
Indeed, this is a classical version of the Lefschetz hyperplane theorem.
To see this, choose a general plane $K$ of dimension $n + 1 - \dim U$.
For any plane codimension-$d$ plane $H \supset K$, we have maps $\pi_1(K \cap U) \ra \pi_1(H \cap U) \ra \pi_1(U)$.
The composition is surjective for general $K$ by applying our theorem in the case that our family is the universal family over the
Grassmannian of $n + 1 - \dim U$ dimensional planes, and therefore $\pi_1(H \cap U) \ra \pi_1(U)$ is surjective for general $H$.

Observe we have a factorization
\begin{equation}
	\nonumber
	\begin{tikzcd} 
		\pi_1(H \cap \sch_b \cap U) \ar {r} \ar {d}{\alpha} &  \pi_1(\sch_b \cap U) \ar {d}{\beta} \\
		\pi_1(H \cap U) \ar {r}{\gamma} & \pi_1(U).
	\end{tikzcd}\end{equation}
Choosing $H$ generally so that $\gamma$ is surjective,
we see $\alpha$ is surjective for a general $b \in B$. Hence, $\beta$ is also surjective for a general $b \in B$.
\end{proof}

We now complete the proof of our main theorem in the case $k = \bc$.
The idea is to find a sufficiently nice open subscheme $B' \subset B$,
and show we can topologically deform one member of this open to another, so that
one member surjects onto a certain finite quotient if and only if all members do.

\begin{proposition}
	\label{proposition:case-over-c}
	\autoref{theorem:hilbert-scheme-galois-image} holds in the case $k = \bc$.
\end{proposition}
\begin{proof}
	By \autoref{lemma:reduction-to-curves}, we may assume the fibers of $\sch \ra B$ have dimension $n + 1 - \dim U$.
	Since $\pi_1$ is a topological invariant, we may assume $\sch$ and $B$ are reduced, by passing to their reductions.
	Further, we can reduce to the case that $B$ is integral by considering each irreducible component of $B$ separately.

In order to show surjectivity of the map $\pi_1(\sch_b \cap U) \ra \pi_1(U)$ for a general $b \in B$,
we will show that for a general $b \in B$, and every connected finite \'etale cover $W \ra U$,
the pullback $W \times_{\bp^n} \sch_b$ is irreducible.
Define $\scc := U \times_{\bp^n} \sch$. Observe we have a map $\scc \ra B$ and $\scc_b = \sch_b \cap U$.
Because the fibers of $\sch \ra B$ have dimension $n + 1 - \dim U$, the generic fiber of $\scc \ra B$ is $1$-dimensional.
The generic fiber is also irreducible by \autoref{corollary:general-pgl-translate} applied in the case that $W \ra U$ 
is the trivial cover.
Note also that $\scc_b \times_{\bp^n} W$ is a connected finite \'etale cover of $\scc_b$ as it is isomorphic to $\scc_b \times_U W$.
To conclude, we will show that for every $W \ra U$ and a general $b \in B$, $\scc_b \times_{\bp^n} W$ is irreducible.

We next construct a dense open in $B' \subset B$. 
Following the construction of $B'$, we will verify that for every $b \in B'$ and every connected finite \'etale connected cover $W \ra U$, $\scc_b \times_{\bp^n} W$ is irreducible.
Loosely speaking, $B'$ will be the locus where the fiber of the normalization of $\scc \ra B$ is the normalization of the fiber.
To construct this $B'$,
let $\overline {\scc}$ denote the projective closure of $\scc \subset B \times \bp^n$.
Let $\widetilde{\overline{\scc}}$ denote the normalization of $\overline{\scc}$, 
let $\widetilde{\scc} := \widetilde{\overline \scc} \times_{\overline \scc} \scc$,
and let $D := \widetilde{\overline \scc} - \widetilde{\scc}$ be the closed subscheme
with reduced subscheme structure.
Define the maps $\xi, \nu,$ and $\rho$ as in the diagram
\begin{equation}
	\nonumber
	\begin{tikzcd}
		D  \ar {r}\ar{rdd}{\xi}  &\widetilde{\overline{\scc}} \ar {d}{\nu} & \widetilde{\scc}  \ar[hook]{l}\ar{d} \\
		\qquad & \overline{\scc} \ar {d}{\rho} & \scc \ar[hook]{l}\ar{dl} \\
		 & B. & 
	 \end{tikzcd}\end{equation}
Since
$\overline{\scc}$ is a finite type $\bc$-scheme,
the normalization map $\nu$ is finite, and therefore $\rho \circ \nu$ is projective. 
Since $\widetilde{\overline{\scc}}$ is normal,
the generic fiber of $\widetilde{\overline \scc} \ra B$ is also normal, hence smooth.
Next, we replace $B$ by a dense open over which $\xi$ is smooth:
To see why such a dense open exists, observe that the generic fiber of $\rho \circ \nu$ is $1$-dimensional and the normalization $\nu$ is birational,
so $\xi$ has $0$-dimensional generic fiber (if it is nonempty).
By generic smoothness, there is a dense open subscheme $D^{\operatorname{sm}} \subset D$ so that $\xi|_{D^{\operatorname{sm}}}$
is smooth, and so $B - \xi(D - D^{\operatorname{sm}})$ determines a dense open of $B$ over which $\xi$ is
smooth.
After shrinking $B$ so that $\xi$ is smooth, take $B' \subset B$ to be a smooth dense open subscheme so that 
$\widetilde{\overline \scc} \times_B B' \ra B'$
is a smooth morphism whose fiber
over each $b \in B'$ is a connected smooth curve (which is possible as the generic fiber of $\scc \ra B$ is a connected smooth curve, as we showed above).
This constructs the desired dense open subscheme $B' \subset B$.

To complete the proof, we will show that
for every $b \in B'$, $\scc_b \times_{\bp^n} W$ is irreducible.
Let $D', \widetilde{\overline \scc}', \widetilde{\scc}', \overline \scc', \scc'$ denote the base changes of
$D, \widetilde{\overline \scc}, \widetilde{\scc}, \overline \scc, \scc$ along $B' \ra B$.
Observe that by construction, each fiber of $\rho \circ \nu$ 
over $b \in B'$ is an integral smooth curve, and therefore $\widetilde{\scc}_b$ is the normalization of $\scc_b$ and $\widetilde{\overline \scc}_b$
is the smooth projective completion of $\widetilde{\scc}_b$.
Since $\widetilde{\scc}_b$ is the normalization of $\scc_b$, we also have
$\widetilde{\scc}_b \times_{\bp^n} W = \widetilde{\scc}_b \times_U W$ is the normalization of $\scc_b \times_{\bp^n} W$.
Hence, in order to show $\scc_b \times_{\bp^n} W$ is irreducible, it suffices to show 
$\widetilde{\scc}_b \times_{\bp^n} W$ is irreducible.

In order to show $\widetilde{\scc}_b \times_{\bp^n} W$ is irreducible, 
we first establish that $\widetilde{\scc}' \ra B'$ is locally trivial in the $C^\infty$ topology.
Since we have assumed that $B'$ is smooth and
$\widetilde{\overline \scc}' \ra B'$ is smooth,
by Ehresmann's fibration theorem, (using crucially $k = \bc$,) the map $\widetilde{\overline \scc}' \ra B'$ is locally trivial in the $C^\infty$ topology.
Furthermore, by construction, $D' \ra B'$ is a proper \'etale map. 
Therefore, again by Ehresmann's theorem, this too is locally trivial.
These facts together imply that $\widetilde{\scc}' = \widetilde{\overline \scc}' - D' \ra B'$ is locally trivial in the $C^\infty$ topology,
since one can pass to a sufficiently small ball in $B'$ over which $\widetilde{\overline \scc}'$ and $D'$ are simultaneously
trivial.

As we have shown above, to complete the proof, it suffices to show that for all $b \in B'$, the image of
$\pi_1(\widetilde{\scc}_b) \ra \pi_1(U)$ surjects onto the finite quotient $\coker \left( \pi_1(W) \ra \pi_1(U) \right)$
corresponding to the connected finite \'etale cover $W \ra U$. 
From \autoref{corollary:general-pgl-translate},
we know there is some $c \in B'$ (in fact $c$ can be chosen generally)
so that $\scc_c \times_{\bp^n} W$ is irreducible.
Since $\widetilde{\scc}_c \times_{\bp^n} W$ is the normalization of $\scc_c \times_{\bp^n} W$,
we know $\widetilde{\scc}_c \times_{\bp^n} W$ is also irreducible. 
Thus, the image of $\pi_1(\widetilde{\scc}_c) \ra \pi_1(U)$ surjects onto $\coker \left( \pi_1(W) \ra \pi_1(U) \right)$.
In the $C^\infty$ topology,
over $k = \bc$, we can choose a smooth path $\gamma$ lying in $B'$
connecting any point $b$ to our particular point $c$.
Upon choosing a local trivialization around each point of $\gamma$, which we have shown to exist in the previous paragraph,
it suffices to show $\pi_1(\widetilde{\scc}_b) \ra \pi_1(U)$ surjects onto 
$\coker \left( \pi_1(W) \ra \pi_1(U) \right)$
when there is a $C^\infty$ local trivialization of $\widetilde{\scc}' \ra B'$ containing both $b$ and $c$.
So, we have reduced to the case that the family is trivial, in which case the image of $\pi_1^{\tp}(\widetilde{\scc}_c) \ra \pi_1^{\tp}(U)$
(where $\pi_1^{\tp}$ denotes the topological fundamental group)
agrees with that of $\pi_1^{\tp}(\widetilde{\scc}_b) \ra \pi_1^{\tp}(U)$. Hence, the images both surject onto $\coker \pi_1(W) \ra \pi_1(U)$,
completing the proof.
Here we are implicitly using that the \'etale fundamental group is the profinite completion of the topological fundamental
group, so a surjection onto a finite quotient in the topological fundamental group implies the same surjection in the \'etale fundamental group.
\end{proof}

\subsection{Reduction to the case $k=\bc$}
\label{subsection:reduction}
We conclude our proof of \autoref{theorem:hilbert-scheme-galois-image}, by deducing the statement over arbitrary fields
of characteristic $0$ from the corresponding statement over $\bc$.
This is accomplished at the end of this section in \autoref{subsubsection:proof-completion}.
Throughout, for $Z$ a scheme over $\spec R$ and $\spec A \ra \spec R$ a map of schemes, we let $Z_A := Z \times_{\spec R} \spec A$.
We start with a lemma that will allow us to reduce to the case of algebraically closed fields.
\begin{lemma}
	\label{lemma:reduction-to-algebraically-closed-fields}
	Under the notations of \autoref{theorem:hilbert-scheme-galois-image},
	suppose $k$ is a field of characteristic $0$ and \autoref{theorem:hilbert-scheme-galois-image} holds over $\overline k$.
	Then it also holds over $k$.
\end{lemma}
\begin{proof}
	Let $\sch \ra B$ be the given $\pgl_{n+1}(k)$-invariant scheme,
and let $\sch_{\overline k} \ra B_{\overline k}$ denote its base change to $\spec \overline k$.
Assuming \autoref{theorem:hilbert-scheme-galois-image} holds over $\overline k$,
there is some dense open
$C^0 \subset B_{\overline k}$ so that for each member $c \in C^0$, $\pi_1((\sch_{\overline k})_c \cap U_{\overline k}) \ra \pi_1(U_{\overline k})$ is surjective.
Let $B^0$ denote the image of $C^0$ under $B_{\overline k} \ra B$. Observe that $B^0 \subset B$ is a dense
open.
We claim that for any $b \in B^0$, the map
$\pi_1(\sch_b \cap U) \ra \pi_1(U)$ is surjective.
To see this, let $c$ be a point of $C^0$ mapping to $b$.
By \cite[Expos\'e IX, Th\'eor\`eme 6.1]{noopsortSGA1Grothendieck1971}, we have a map of exact sequences
\begin{equation}
	\nonumber
	\begin{tikzcd}
		0 \ar {r} & \pi_1((\sch_b \cap U)_{\overline k}) \ar {r} \ar {d} & \pi_1(\sch_b \cap U) \ar {r} \ar {d} & \pi_1(\spec k) \ar {r} \ar {d} & 0 \\
		0 \ar {r} & \pi_1(U_{\overline k}) \ar {r} & \pi_1(U) \ar {r} & \pi_1(\spec k) \ar {r} & 0.
	\end{tikzcd}\end{equation}

Since we have assumed 
$\pi_1((\sch_b \cap U)_{\overline k})=\pi_1((\sch_{\overline k})_c \cap U_{\overline k}) \ra \pi_1(U_{\overline k})$ is surjective, then the above diagram implies
$\pi_1(\sch_b \cap U) \ra \pi_1(U)$ is surjective.
\end{proof}
Although the following is surely well known and standard, 
for completeness,
we record the following lemma for verifying
surjectivity of a map on $\pi_1$. 
\begin{lemma}
\label{lemma:base-change-surjective-implies-surjective}
Suppose $X \ra Y$ is a map of finite type schemes over an algebraically
closed field $k$ of characteristic $0$ and let $L \supset k$ be another algebraically
closed field. Then, the map $\pi_1(X) \ra \pi_1(Y)$ is surjective if and only
if $\pi_1(X_L) \ra \pi_1(Y_L)$ is surjective.
\end{lemma}
\begin{proof}
	First, observe we have a commutative square
	\begin{equation}
		\label{equation:surjectivity-base-change-factoring-out}
		\begin{tikzcd} 
			\pi_1(X_L) \ar {r} \ar {d} & \pi_1(Y_L) \ar {d} \\
			\pi_1(X) \ar {r} & \pi_1(Y).
		\end{tikzcd}\end{equation}
	Note that the vertical maps of \eqref{equation:surjectivity-base-change-factoring-out}
	are surjective because connected schemes
	remain connected upon base change between algebraically closed fields.
	Therefore, if the top map of \eqref{equation:surjectivity-base-change-factoring-out}
is surjective, the bottom is as well.
	It only remains to show the bottom map is surjective if the top
	map is.

	One can immediately deduce this from the somewhat tricky
	fact that the vertical maps of \eqref{equation:surjectivity-base-change-factoring-out}
are isomorphisms, as shown in \cite[Theorem 1.1]{landesman:invariance-of-the-fundamental-group}.

Instead, we opt for a more direct proof.
To show the top map of \eqref{equation:surjectivity-base-change-factoring-out}
is surjective, it suffices to show that for any
connected finite \'etale cover $W_L \ra Y_L$, the pullback
$W_L \times_{Y_L} X_L$ remains connected.
By writing $L$ as the limit of affine $k$-algebras,
we can spread out the cover $W_L \ra Y_L$
to a finite \'etale cover $W_A \ra Y_A$ for
$A$ a smooth affine $L$-subalgebra over $k$,
so that $W_L = W_A \times_{\spec A} \spec L$.

We will show $W_L \times_{Y_L} X_L$ is connected.
Because $W_L$ is geometrically connected, it follows that the generic fiber
of the map $W_A \ra \spec A$ is geometrically connected.
Therefore, by constructibility of the geometrically connected locus,
\cite[Th\'eor\`eme 9.7.7(ii)]{EGAIV.3},
there is a dense open set of closed points
$b \in \spec A$
so that the fibers of the map
$W_A \ra A$ are geometrically connected.
Since 
$\pi_1(X) \ra \pi_1(Y)$
is surjective,
for $b \in \spec A$ with $W_A \times_{\spec A} b$ geometrically connected,
the pullback
$(W_A \times_{\spec A} b) \times_Y X \simeq W_A \times_{Y_A} X_A \times_{\spec A} b$
is also geometrically connected.
In turn, by 
constructibility of the geometrically connected locus,
\cite[Th\'eor\`eme 9.7.7(ii)]{EGAIV.3},
it follows that the generic fiber of
$W_A \times_{Y_A} X_A \ra \spec A$ is geometrically connected.
Hence, $W_L \times_{Y_L} X_L$ is connected,
being the base change of the generic fiber of
$W_A \times_{Y_A} X_A \ra \spec A$ to $\spec L$.
\end{proof}

We now prove a lemma which will let us deal with algebraically closed fields contained in $\bc$.
\begin{lemma}
	\label{lemma:reduction-to-subfield}
Suppose  $k \subset L$ are two algebraically closed fields of characteristic $0$. If 
	\autoref{theorem:hilbert-scheme-galois-image}
	holds over
	$L$ then it also holds over $k$.
\end{lemma}
\begin{proof}
Let $\sch \ra B$ denote our given family over $k$ and let $C^0 \subset B_L$ denote a dense open 
subscheme given by \autoref{theorem:hilbert-scheme-galois-image}
so that for $c \in C^0$, $\pi_1((\sch_L)_c \cap U_L) \ra \pi_1(U_L)$
is surjective.
Let $B^0$ denote the image of $C^0$ in $B$.
Say $c \in C^0$ maps to a point $b \in B^0$. Then, under the identification
$(\sch_L)_c \cap U_L \simeq (\sch_b \cap U)_L$,
we obtain a commutative square
\begin{equation}
	\label{equation:square-algebraically-closed-base-change-from-below}
	\begin{tikzcd} 
		\pi_1((\sch_b \cap U)_L) \ar {r} \ar {d} & \pi_1(U_L) \ar {d} \\
		\pi_1(\sch_b \cap U) \ar {r} & \pi_1(U).
	\end{tikzcd}\end{equation}
Since the top map is surjective by assumption, it follows from 
\autoref{lemma:base-change-surjective-implies-surjective}
that the bottom map is then surjective for all $b \in B^0$,
completing the proof.
\end{proof}

We next prove a lemma that will let us deal with fields containing $\bc$.
\begin{lemma}
	\label{lemma:to-bigger-field}
	Suppose $L$ is an algebraically closed field of characteristic $0$ with an inclusion $\bc \ra L$. If 
	\autoref{theorem:hilbert-scheme-galois-image}
	holds over
	$k = \bc$ then it also holds over $L$.
\end{lemma}
\begin{proof}
	Retain the notation of \autoref{theorem:hilbert-scheme-galois-image}.
	By spreading out, we can find a finitely generated $\bz$-algebra $A$, a morphism $\sch_A \ra B_A$ over $A$, and a scheme $U_A$ over $A$ so that
$\sch \ra B$ and $U$ are the base changes of $\sch_A \ra B_A$ and $U_A$ along the map $\spec L \ra \spec A$.
Choosing inclusions $A \ra \bc \ra L$, we obtain $\sch_\bc \ra B_\bc$ and $U_\bc$ so that
$\sch \ra B$ and $U$ are the base change 
$\sch_\bc \ra B_\bc$ along $\spec L \ra \spec \bc$.
Assuming \autoref{theorem:hilbert-scheme-galois-image} holds over $\bc$,
then we can find some dense open $B^0_\bc \subset B_\bc$ 
so that for any $b \in B^0_\bc$, the map $\pi_1(\sch_b \cap U_\bc) \ra \pi_1(U_\bc)$ is a surjection.
Then, for any $c \in (B^0_\bc)_L$, let $b \in B^0_\bc$ be the image of $c$ along $(B^0_\bc)_L \ra B^0_\bc$.
Because $\sch_c \cap U = (\sch_b \cap U_\bc)_L$, we obtain a commuting square
\begin{equation}
	\label{equation:square-algebraically-closed-base-change}
	\begin{tikzcd} 
		\pi_1((\sch_b \cap U_\bc)_L) \ar {r} \ar {d} & \pi_1(U) \ar {d} \\
		\pi_1(\sch_b \cap U_\bc) \ar {r} & \pi_1(U_\bc).
	\end{tikzcd}\end{equation}
Since the bottom map is surjective by assumption, it follows from
\autoref{lemma:base-change-surjective-implies-surjective}
that the top map is surjective.
\end{proof}

\subsubsection{Proof of \autoref{theorem:hilbert-scheme-galois-image}}
\label{subsubsection:proof-completion}
Combining the above lemmas with the knowledge that our main theorem holds over $k = \bc$,
we complete the proof of \autoref{theorem:hilbert-scheme-galois-image}.

\begin{proof}[Proof of \autoref{theorem:hilbert-scheme-galois-image}]
	By \autoref{proposition:case-over-c}, \autoref{theorem:hilbert-scheme-galois-image} holds if $k = \bc$.
	By \autoref{lemma:to-bigger-field}, 
	the main theorem then holds over any algebraically closed field containing $\bc$.
	Then, by \autoref{lemma:reduction-to-subfield},
\autoref{theorem:hilbert-scheme-galois-image}
	holds over any algebraically closed field of characteristic $0$,
	because every algebraically closed field of characteristic $0$
	has an injection into some algebraically closed field containing
	$\bc$.
	Finally, if $k$ is any field of characteristic $0$, since 
	\autoref{theorem:hilbert-scheme-galois-image}
	holds over $\overline k$,
	it follows from
	\autoref{lemma:reduction-to-algebraically-closed-fields}
	that
	\autoref{theorem:hilbert-scheme-galois-image}
	holds over $k$.
\end{proof}

\section{Acknowledgements}
I thank David Zureick-Brown for introducing me to the subject of Galois representations, which led to the genesis of this paper.
I thank David Zywina for writing \cite[Lemma 5.2]{zywina2010hilbert}, which was the inspiration for this paper, and
James Tao for help understanding the proof of \cite[Lemma 5.2]{zywina2010hilbert}.
I thank J\'anos Koll\'ar and Wenhao Ou for several helpful comments.
I would like to thank Ravi Vakil for listening to my argument and helping me fix several technical points.
I thank Jason Starr for pointing out past work on this subject, particularly J\'anos Koll\'ar's papers. I thank Brian Conrad and Jason Starr
for teaching me about properties
of the \'etale fundamental group upon base change.
I would to thank Fran\c cios Charles, Dougal Davis, 
Laurent Moret-Bailly,
Arpon Raksit, Eric Riedl, Zev Rosengarten, 
and Ashvin Swaminathan for helpful conversations.
This material is based upon work supported by the National Science Foundation Graduate Research Fellowship Program under Grant No. DGE-1656518. 

\bibliographystyle{alpha}
\bibliography{/home/aaron/Dropbox/master}

\end{document}